\documentclass[11pt]{amsproc}

\usepackage{amsmath, amsfonts, amsthm, amssymb, anysize}
\usepackage{url}
\usepackage{tikz,tkz-graph}
\usetikzlibrary{calc}

\usepackage{hyperref}
\hypersetup{citecolor=red, linkcolor=blue, colorlinks=true}

\newtheorem{theorem}{Theorem}[section]
\newtheorem{lemma}[theorem]{Lemma}
\newtheorem{proposition}[theorem]{Proposition}
\newtheorem{corollary}[theorem]{Corollary}
\newtheorem{conjecture}[theorem]{Conjecture}
\newtheorem{remark}[theorem]{Remark}
\newtheorem{definition}[theorem]{Definition}

\newcounter{claim_nb}[theorem]
\newcommand\reref[1]{Remark~\ref{re:#1}}

\newcommand{\ca}{\alpha^\circ}
\newcommand{\myc}{\mathcal M}
\DeclareMathOperator{\girth}{\mathsf{girth}}

\renewcommand\le{\leqslant}
\renewcommand\ge{\geqslant}

\title{The circular altitude of a graph}

\author[Bamberg]{John Bamberg}
\author[Corr]{Brian Corr}
\author[Devillers]{Alice Devillers}
\author[Hawtin]{Daniel Hawtin}
\author[Pivotto]{Irene Pivotto}
\author[Swartz]{Eric Swartz}
\address[Bamberg, Devillers, Hawtin, Pivotto]{Centre for the Mathematics of Symmetry of Computation,
School of Mathematics and Statistics, University of Western Australia, Perth WA, Australia.}
\address[Corr]{Departamento de Matem\'atica, Instituto de Ci\^encias Exatas, Universidade Federal
de Minas Gerais, Av. Ant\^onio Carlos, 6627, 31270-901 Belo Horizonte, MG, Brazil.}
\address[Swartz]{Department of Mathematics, College of William and Mary, P.O. Box 8795, Williamsburg, VA 23187, U.S.A.}
\email{john.bamberg@uwa.edu.au, brian.p.corr@gmail.com, alice.devillers@uwa.edu.au, \newline daniel.hawtin@research.uwa.edu.au, irene.pivotto@uwa.edu.au, easwartz@wm.edu}

\begin{document}

\maketitle

\begin{abstract}
In this paper we investigate a parameter of graphs, called the circular altitude, introduced by Peter Cameron. We show that the circular altitude provides a lower bound on the circular chromatic number, and hence on the chromatic number, of a graph and investigate this parameter for the iterated Mycielskian of certain graphs. 
\end{abstract}

\section{Introduction}

The \emph{linear altitude} $\alpha(G)$ of an undirected and simple graph $G$ is the largest guaranteed length of a
monotonic path in any vertex-ordering. That is, it is the largest integer $k$ such that for any ordering of the 
vertices of $G$, there is a monotonic path with $k$ vertices.
It turns out that $\alpha(G)$ is equal to the chromatic number $\chi(G)$ of $G$: 
\begin{itemize}
\item If we have a $k$-colouring of $G$, then ordering the vertices according to their colour, and then vertices within the same colour class arbitrarily, produces a linear ordering where every monotonic path has length at most $k$, hence $\alpha(G)\le \chi(G)$;
\item Conversely, suppose we have a linear ordering of the vertices of $G$.
Then for every vertex $v$, assign a colour $c(v)$ by taking the length of the the largest increasing path ending at $v$. This
assignment produces a proper colouring, hence, $\alpha(G)\ge \chi(G)$.
\end{itemize}

Peter Cameron proposed an analogue of the linear altitude for circular orderings \cite{PJC}.
Given a circular ordering $\sigma$ of the vertices $V(G)$ of $G$ (that is, a one-to-one assignment of $V(G)$ to the vertices of a directed cycle), 
we say that a cycle $u_1,u_2,\ldots,u_n$ is {\em monotonic} for $\sigma$ if $u_1,\ldots,,u_n, u_1$ appear in this circular order in $\sigma$.
The {\em circular altitude} of a graph $G$, denoted $\ca(G)$, is defined as the largest $k$ such that every circular ordering of the vertices of $G$ contains a 
monotonic cycle of length at least $k$. For this definition we consider edges as monotonic cycles of length two, so the circular altitude is at least $2$ for any graph with at least one edge.

If $H$ is a subgraph of $G$, then any circular ordering of $V(G)$ induces a circular ordering of $V(H)$, so $\ca(G) \ge \ca(H)$.
In particular, $\ca(G) \ge \omega(G)$, the clique number of $G$.
Now suppose that $U_1, \ldots, U_s$ is a partition of $V(G)$ into stable sets (also known as \emph{cocliques} or \emph{independent sets}) and consider any circular ordering of $V(G)$ having first all the vertices in $U_1$, then all the vertices in $U_2$, and so on appearing in this order. Every monotonic cycle for this ordering may use at most one vertex for each of the $U_i$'s, so $\ca(G) \le s$. In particular, $\ca(G) \le \chi(G)$, the chromatic number of $G$.
If $\ca(G)$ is not $2$, then we also have $\ca(G) \ge \girth(G)$.
We wish to investigate which graphs achieve the above bounds, and whether there exist graphs $G$ satisfying $\omega(G)< \ca(G) < \chi(G)$.

Given a graph $G$, the vertex set of $\myc(G)$ is defined to be
\[ V(\myc(G))=\{w\} \cup \{a^u \ |\ a \in V(G)\} \cup \{a^v \ |\ a \in V(G)\}.\]
The induced subgraph on the vertices $\{a^v \ |\ a \in V(G)\}$ is a copy of $G$.
Vertex $w$ is adjacent to all vertices of the form $a^u$, and for every edge $ab$ in $G$ we add edges $a^ub^v$ and $b^ua^v$ to $\myc(G)$. For example, $\myc(K_2)=C_5$ and $\myc(C_5)$ is the Gr\"otzsch graph.
We denote by 
$\myc^{d}(G)$ the graph obtained from $G$ by iterating this process $d$ times (where $\myc^0(G)=G$).
This construction was defined by Jan Mycielski in~\cite{Mycielski:1955aa} and has the property that $\chi(\myc(G)) = \chi(G)+1$ and $\omega(\myc(G))=\omega(G)$. 
We are interested in graphs which satisfy $\omega(G)< \ca(G) < \chi(G)$, and Mycielskians provide promising examples. 
In fact, our first main result is the following.

\begin{theorem}\label{thm:oddgirth}
Suppose $G$ is nonempty with $\chi(\myc^r(G))=t$, where $t$ is odd, and the length of the shortest odd cycle of $G$ (i.e., its `odd girth') is strictly greater than $t$. Then 
$\ca(\myc^r(G))<t$.
\end{theorem}

As a corollary of Theorem~\ref{thm:oddgirth}, we obtain the following, where $C_{2n+1}$ denotes an odd cycle of length $2n+1$.

\begin{theorem}\label{thm:camyc}\leavevmode
If $n > r+1$, then $\omega(\myc^{2r}(C_{2n+1}))=2$,
$\ca(\myc^{2r}(C_{2n+1}))=2r+2$, and $\chi(\myc^{2r}(C_{2n+1}))=2r+3$.
\end{theorem}

Our second main result concerns the relation between the circular altitude and the \emph{circular chromatic number} of a graph.
The circular chromatic number $\chi_c(G)$ is often considered to be the most natural generalisation of the chromatic number of a graph $G$. Briefly, it
is the infimum of the magnitudes of the circular colourings of $G$ (see Section \ref{section:circchrom} for definitions) and it satisfies the bound
$\chi(G)-1< \chi_c(G)\le \chi(G)$  (see \cite[Theorem 1.1]{Zhu:2001aa}).
In this paper, we show that:

\begin{theorem}\label{th:caVScirc}
$\ca(G)\le\chi_c(G)$.
\end{theorem}

We should emphasise that the value of $\ca(G)$ is not dependent on other known lower bounds
for $\chi_c(G)$.

In the theory of circular colouring, one prominent open problem is the determination of $\chi_c(\myc^3(K_6))$, the circular
chromatic number of the third Mycielskian of the complete graph $K_6$.
The Zig-zag Theorem (discussed in Section~\ref{sec:myc}) implies that $\ca(\myc^3(K_5))=\ca(\myc^2(K_6))=8$ and so $8\le \ca(\myc^3(K_6))$. On the other hand, $\chi(\myc^3(K_6))=9$ and so
\[
8\le \ca(\myc^3(K_6))\le \chi_c(\myc^3(K_6))\le \chi(\myc^3(K_6))=9.
\]

After running trial computations by computer, we are confident in the following conjecture: 

\begin{conjecture}
$\ca(\myc^3(K_6))=9$ and hence $\chi_c(\myc^3(K_6))=9$.
\end{conjecture}

On closer inspection, it appears that the circular altitude is more than just a surface generalisation of the concept of linear ordering. Indeed, the connection with the circular chromatic number and the surprising link between circular altitude and the very powerful topological techniques of the Zig-zag Theorem (see Section~\ref{sec:myc}) demonstrate its utility.

\section{Basic definitions and notation}

Let $G$ be a graph. For vertices $a,b$ the notation $a \sim b$ means that vertices $a,b$ are adjacent in $G$, while we indicate the complement of $G$ as $\overline{G}$. The length of a path in $G$ will be taken to be the number of vertices in the path.

Given a linear ordering of $V(G)$ we may obtain a proper colouring of $V(G)$ as follows: for every vertex $v \in V(G)$, the colour of $v$ is the number of vertices in a longest monotonic path ending at $v$. We call such colouring the colouring {\em induced} by the linear ordering. The following result follows easily from this definition. 

\begin{remark}\label{re:lin-col}
If $c$ is the colouring of $V(G)$ induced by a linear ordering of $V(G)$ and $u$ is adjacent to $v$ in $G$, then $c(u) < c(v)$ if and only if $u$ precedes $v$ in the linear ordering.
\end{remark}

Consider a circular ordering $u_1,u_2,\ldots,u_n$ of $V(G)$; for every $u_i$, such circular ordering induces two linear orderings of $V(G)$ starting at $u_i$, one as $u_i < u_{i+1} < \cdots < u_n < u_1 < \cdots < u_{i-1}$ and the other as $u_i < u_{i-1} < \cdots < u_1< u_n < \cdots <u_{i+1}$. 
For our purposes it is irrelevant which one of the linear orderings we pick. The next result follows from \reref{lin-col}.

\begin{remark}
Consider a circular ordering of $V(G)$ and a linear ordering starting at $u \in V(G)$ induced by this circular ordering. Let $c$ be the colouring of $V(G)$ induced by the linear ordering and suppose that $C=\{u_1,\ldots,u_k\}$ is a cycle of $G$ with $c(u_1) < c(u_2) < \cdots < c(u_k)$. Then $C$ is a monotonic cycle for the initial circular ordering.
\end{remark}

\section{Connection between circular altitude and circular chromatic number}\label{section:circchrom}

The \textit{circular chromatic number} $\chi_c(G)$ of $G$ is the infimum over all real numbers $r$ such that there exists a map from $V(G)$ to a circle of circumference $1$ with the property that any two adjacent vertices map to points at distance at least $1/r$ apart along this circle.
An equivalent definition is given below in terms of circular colourings (see \cite[Section 2]{Zhu:2001aa}).

\begin{definition}
For positive integers $p$ and $q$ a colouring $c: V(G) \rightarrow \{1,\ldots,p\}$ of a graph $G$ is called a \textit{(p,q)-colouring} if for all adjacent vertices $u$ and $v$ one has $q \le |c(u) - c(v)| \le p-q.$  The \textit{circular chromatic number} of $G$ is defined as
$$\chi_c(G):= \inf\left\{\frac{p}{q}: \text{ there is a } (p,q)\text{-colouring of } G\right\}.$$ 
\end{definition}

It is known for every graph $G$ that $\chi(G) - 1 < \chi_c(G) \le \chi(G)$ (see \cite[Theorem 1.1]{Zhu:2001aa}).  Given the second definition/interpretation of $\chi_c(G)$, it seems like there should be a connection between the circular altitude $\ca(G)$ and the circular chromatic number $\chi_c(G)$.  Indeed, there is such a connection, as stated in Theorem~\ref{th:caVScirc}:
\textit{For every graph $G$, $\ca(G) \le \lfloor \chi_c(G) \rfloor$.}

\begin{proof}[Proof of Theorem~\ref{th:caVScirc}]
Let $\ca(G) = t$ and let $c:V(G) \rightarrow \{1,\ldots,p\}$ be any $(p,q)$-colouring of $G$.   By definition, for all $u \sim v$, $q \le |c(u) - c(v)| \le p - q.$  We convert this colouring into a circular ordering.  We start by creating a linear order, where we start with all vertices with colour 1 (in any order among themselves), then we place all vertices with colour 2, then all with colour 3, etc., until all vertices have been placed.  We turn this into a circular ordering by placing the very first vertex in our linear order and placing vertices clockwise.

Since $\alpha^o(G) = t$, for some integer $m \ge t$, there are vertices $v_1, v_2, \ldots, v_m$ such that $v_1 \sim v_2$, $v_2 \sim v_3$, \ldots, $v_{m-1} \sim v_m$, $v_m \sim v_1$, and $c(v_1) < c(v_2) < \ldots < c(v_m)$.  By the definition of $(p,q)$-colouring, for all $1 \le i \le m-1$ we have $c(v_{i+1}) - c(v_i) \ge q$, and so $c(v_m) - c(v_1) \ge (m-1)q \ge (t-1)q.$  On the other hand, by the definition of $(p,q)$-colouring, $c(v_m) - c(v_1) \le p-q$.  Hence $(t-1)q \le p - q,$ and so $\ca(G) = t \le p/q.$  This is true for any $(p,q)$-colouring, and so $\ca(G) \le \chi_c(G)$.  Since $\ca(G)$ is an integer, we get the desired result.  
\end{proof}

\begin{corollary}
If $\ca(G)>2$ then $\chi_c(G)\ge \girth(G)$.
\end{corollary}

\section{Criteria for $\chi_c(G) = \chi(G)$}

We have the following immediate corollary of Theorem~\ref{th:caVScirc}:

\begin{corollary}
\label{cor:circchrom}
If $\ca(G) = \chi(G)$, then $\chi_c(G) = \chi(G)$. 
\end{corollary}

Section 3 of \cite{Zhu:2001aa} is all about conditions for when $\chi_c(G) = \chi(G)$. 
In particular, one result in~\cite{Zhu:2001aa} is the following.

\begin{theorem}[Corollary 3.1 in \cite{Zhu:2001aa}]\label{disconnected}
If $\overline{G}$ is disconnected, then $\chi_c(G)=\chi(G)$.
\end{theorem}

We extend the result of Theorem \ref{disconnected} to the circular altitude.

\begin{proposition}
\label{prop:compdisconn}
If $\overline{G}$ is disconnected,  then $\ca(G) = \chi(G)$. 
\end{proposition}

\begin{proof}
Let $A$ be the vertex set of a connected component of $\overline{G}$, and let $B=V(G)\setminus A$. We adopt the notation $G[A]$ for the induced subgraph of $G$ with vertex set $A.$  Note that since $A$ is a connected component of $\overline{G}$, for all vertices $u \in A$ and $v \in B$, $u \sim v.$  From this, it is immediate that if $\chi(G[A]) = r$ and $\chi(G[B]) = s$, then $\chi(G) = r+s.$

Now assume that we are given a circular ordering of the vertices of $G.$  Choose two vertices $a,b$ in the ordering that occur consecutively such that $a \in A$, $b \in B$, and $b$ comes before $a$ in the ordering if moving in the clockwise direction.  We now induce a linear ordering on the vertices by starting with the vertex $a$ and continuing as in the circular ordering, ending at $b$.  Note that since $\chi(G[A]) = r$ and $\chi(G[B]) = s$, we have $\alpha(G[A]) = r$ and $\alpha(G[B]) = s$, and there exists a monotonic path $P_{A}$ (respectively $P_{B}$) that uses only vertices from $A$ (respectively $B$) of length at least $r$ (respectively at least $s$).  We now combine these two paths together.  Since every vertex in $A$ is adjacent to every vertex in $B$, this creates one long monotonic path $P$ of length at least $r+s$.  If $P$ begins with a vertex of $A$ and ends with a vertex of $B$ (or vice versa), then $P$ is actually a cycle of length at least $r+s$ in the ordering, and we are done.  Otherwise, $P$ begins and ends with two vertices either from $A$ or from $B$.  If it is two vertices from $A$, then these are both adjacent to the vertex $b$ at the end of the ordering and we have a cycle of length at least $r+s+1$.  If it is two vertices from $B$, then these are both adjacent to the vertex $a$ at the beginning of the ordering and we have a cycle of length at least $r+s+1$.  In any case, this circular ordering has a monotonic cycle of length at least $r+s$.  Since $r+s$ is the chromatic number for $G$ (and upper bound for $\ca(G)$), $\ca(G) = \chi(G)$, as desired.
\end{proof}

The above \cite[Corollary 3.1]{Zhu:2001aa} is implied by the following. 

\begin{theorem}[{\cite[Theorem 3.1]{Zhu:2001aa}}]
\label{thm:chic}
Suppose $\chi(G) = m.$  If there is a proper nonempty subset $A$ of $V(G)$
such that for any $m$-colouring $c$ of $G$, each colour class $X$ of $c$ is either contained in $A$ or is disjoint from $A$, then $\chi_c(G) = \chi(G)$. 
\end{theorem}

It was noted in \cite{Zhu:2001aa} that all other known sufficient conditions for a graph $G$ to satisfy $\chi_c(G) = \chi(G)$ were easily derived from Theorem \ref{thm:chic}.  However, Corollary \ref{cor:circchrom} is a legitimately distinct sufficient condition.  For instance, as noted on \cite[p. 378]{Zhu:2001aa}, for any integers $n \ge 1$ and $g \ge 3$, there is a graph of girth at least $g$ that is uniquely $n$-colourable.  For this graph, Theorem \ref{thm:chic} implies that the circular chromatic number is $n$, but Corollary \ref{cor:circchrom} is useless if $g > n$.  On the other hand, if $G=\myc(K_3)$ is the Mycielskian of $K_3$ (see Figure \ref{MK3}) then $G$ has vertex set $\{a^u,a^v,b^u,b^v,c^u,c^v,w\}$ and edge set 
\[
\{a^vb^v, a^vc^v, b^vc^v, a^ub^v, a^uc^v, b^ua^v, b^uc^v, c^ua^v, c^ub^v, a^uw, b^uw,  c^uw\}.
\]
By a corollary of the Zig-zag Theorem in \cite{Simonyi:2006aa} (see Theorem \ref{thm:zigzag} in the next section), $\ca(G) = 4$, and hence $\ca(G) = \chi_c(G) = \chi(G) = 4$.  (It should be noted that one can prove that $\ca(G) = 4$ for this graph by elementary methods as well.)  On the other hand, there exist colourings $c_1: V(G) \rightarrow \{1,2,3,4\}$ and $c_2: V(G) \rightarrow \{1,2,3,4\}$ given by $c_1(w) = 1$, $c_1(a^v) = c_2(a^u) = 2$, $c_1(b^v) = c_1(b^u) = 3$, $c_1(c^v) = c_1(c^u) = 4$, and $c_2(w) = c_2(a^v) = 1$, $c_2(a^u) = c_2(b^u) = c_2(c^u) = 2$, $c_2(b^v) = 3$, $c_2(c^v) = 4$, which demonstrate that $G$ does not satisfy the hypotheses of Theorem \ref{thm:chic}.  This leads to the following ``natural'' question: can Corollary \ref{cor:circchrom} be used to show that $\chi_c(G) = \chi(G)$ for any new graphs $G$?  For instance, one such open case is $G = \myc^t(K_n)$ where $n \ge t+2$ and $n+t$ is odd (the even case was settled by Simonyi and Tardos in \cite{Simonyi:2006aa}).

\begin{figure}[ht]
\caption{$\myc(K_3)$}
\begin{tikzpicture}[scale=0.8]
   \pgfmathsetmacro{\xcoord}{2*cos(210)}%
    \pgfmathsetmacro{\ycoord}{2*sin(210)}%
    \pgfmathsetmacro{\xcoordtwo}{2*cos(330)}%
\Vertex[x=\xcoord, y=\ycoord, L = $a^u$ ]{A}  
\Vertex[x=2*\xcoord, y=2*\ycoord, L = $a^v$]{A2}  
\Vertex[x=\xcoordtwo,y=\ycoord, L = $b^u$]{B}
\Vertex[x=2*\xcoordtwo,y=2*\ycoord, L = $b^v$]{B2}
\Vertex[x=0,y=2, L = $c^u$]{C}
\Vertex[x=0,y=4, L = $c^v$]{C2}
\Vertex[L=$w$]{O}
\Edges(A2,B2,C2,A2)
\Edges(B2,A,O,B,A2,C,O)
\Edges(B,C2,A)
\Edges(B2,C)
\end{tikzpicture}\label{MK3}
\end{figure}

\section{The circular altitude of iterated Mycielskians}\label{sec:myc}

First we need to introduce some notation which will be convenient for dealing with iterated Mycielski graphs.
Recall that the Mycielskian $\myc(G)$ of a graph $G$ is the graph obtained from $G$ by adding a vertex $u'$ for every $u \in V(G)$, plus an extra vertex $w$ so that every vertex $u'$ is adjacent to the neighbours of $u$ and to $w$ (and no further edges are added to the graph). Iterating this process $i$ times we obtained the graph $\myc^i(G)$.

Given a graph $G$ we iteratively define (for every $i \ge 1$) the vertex set of $\myc^i(G)$ to be 
\[ V(\myc^i(G))=\{w_i\} \cup \{a^u \ |\ a \in V(\myc^{i-1}(G))\} \cup \{a^v \ |\ a \in V(\myc^{i-1}(G))\},\]
where the vertices of the form $a^v$ correspond to the original vertices of $\myc^{i-1}(G)$, vertex $a^u$ is the one paired-up with $a^v$ and $w_i$ is the vertex adjacent to all the other new vertices.
From the definition of Myciesklian, in this notation we have the following adjacency rules:
 \begin{itemize}
	\item[(R1)] For any vertex $a \in V(\myc^{i-1}(G))$, $w_i \sim a^u$ and $w_i \not\sim a^v$.
	\item[(R2)] For any vertices $a,b \in V(\myc^{i-1}(G))$, $a^u \not\sim b^u$.
	\item[(R3)] For any vertices $a,b \in V(\myc^{i-1}(G))$, $a^v \sim b^v \Leftrightarrow a^v \sim b^u \Leftrightarrow a \sim b.$ 
\end{itemize}

So, for example, if $G\cong K_2$, with vertex set $\{a,b\}$, then $\myc(G)$ will be the $5$-cycle with vertex set $\{a^u,a^v,b^u,b^v,w_1\}$ and edge set 
\[
\{a^vb^v, a^ub^v, b^ua^v, a^uw_1, b^uw_1 \}.
\]
The vertices of $\myc^2(G)$ are $w_2, w_1^u, w_1^v$ together with all vertices of the form $a^W, b^W$, or  $c^W$, for all choices of a word $W$ of length two in $u,v$.
It turns out that $\myc^2(G)$ is isomorphic to the \emph{Gr\"otzsch graph}.

We will add the following conventions: when used in exponential notation, $[m]$ refers to an arbitrary word of length $m$ in $u,v$, and we will always read words from right to left.  
So, for instance, $[r-i]u[i-1]$ is an arbitrary word in $u,v$ that has length $r$ and has a $u$ in the $i^{th}$ position.  

\begin{lemma}
\label{lem:w_i}
Let $W_1$ and $W_2$ be any two distinct words of length $r-i$ in $u,v$.  Then:
\begin{enumerate}
\item[(i)] If $i^{W_1} \sim j^{W_2}$ in $\myc^r(G)$, then $i\sim j$ in $G$, unless $i$ or $j$ is a ``$w$'' vertex.
\item[(ii)] If, for some $1\le s \le r$, both $W_1$ and $W_2$ have a $u$ in position $s$, then $i^{W_1}$ is not adjacent to $j^{W_2}$ in $\myc^r(G)$, for every $i,j \in V(G)$.
\item[(iii)] $w_i^{W_1}$ is not adjacent to $w_i^{W_2}$.
\end{enumerate}
\end{lemma}

\begin{proof}
By induction, the adjacency rules R1--R3 imply (i) and (ii).
Finally (iii) follows from the definition of the Mycielski construction (no vertex is ever adjacent to a copy of itself) and rules R1--R3 for adjacency. 
\end{proof}

For ease of notation, we will let $\myc^0(G) = G$.  Given a (linear) ordering of the vertices of $\myc^r(G)$, $r \ge 1$, the vertices of $\myc^{r-1}(G)$ correspond to the vertices in $i^W$, where $i \in V(G)$ and $W$ is a word whose rightmost letter is $v$, and $w_j^U$, where $j \le r-1$ and $U$ is a word of length $r - j$ whose rightmost letter is $v$.  This subset of vertices induces an ordering of $\myc^{r-1}(G)$.  Based on this observation, we will call this ordering of $\myc^{r-1}(G)$ the {\em ordering of $\myc^{r-1}(G)$ inherited from the ordering of $\myc^r(G)$}.  Note that, starting with an ordering of $\myc^r(G)$, we may iterate this process until we arrive at an induced order for $\myc^0(G) = G$.  This order corresponds to a colouring $c$ of the original graph, and we refer to this as the {\em colouring of $G$ induced by the ordering of $\myc^r(G)$}.

We now present a convention for ordering the vertices of $\myc^r(G)$, where $G$ is any graph. 
We order the words of length $r$ in $u,v$ lexicographically by reading from right to left and assuming that $v < u.$  (So $uvv < vuv < uvu$, for instance.)  A \textit{powerful ordering} of $\myc^r(G)$ is a linear ordering of the vertices satisfying all of the following five conditions:
\begin{itemize}
	\item[(P0)] The colouring $c$ of $G$ induced by the powerful ordering of $\myc^r(G)$ is a $\chi(G)$-colouring of $G$.
	\item[(P1)] For any $i,j \in V(G)$ and any fixed word $W$ of length $r$ in $u,v$, if $i^W$ occurs before $j^W$ in the ordering, then $c(i) \le c(j)$.
	\item[(P2)] For any words $W_1$ and $W_2$ of length $r$ in $u,v$ and any $i,j \in V(G)$, if $W_1 < W_2$ lexicographically, then $i^{W_1}$ occurs before $j^{W_2}$ in the ordering.
	\item[(P3)] For all $1 \le i \le r$ and all $j \in V(G)$, $w_i^{[r-i]}$ occurs before $j^{[r]}$ in the ordering.
	\item[(P4)] For all $i < j$, $w_i^{[r-i]}$ occurs before $w_j^{[r-j]}$ in the ordering. 
\end{itemize}

\begin{lemma}\label{lemma:powerful}
Let $c\colon V(G) \rightarrow \{1,\ldots, d\}$ be a proper colouring of a graph $G$, where $d=\chi(G)$. Then $\myc^r(G)$ has a powerful ordering inducing the colouring $c$.
\end{lemma}
\begin{proof}
%
It is easy to obtain a powerful ordering of $V(\myc^r(G))$ inducing this colouring as follows. First place all the vertices of the form $w_1^{[r-1]}$ (in any order), then all vertices of the form $w_2^{[r-2]}$ (in any order) and so on up to vertex $w_r$. This guarantees that properties P3 and P4 of powerful orderings are satisfied. Next, order all words of length $r$ in $u,v$ lexicographically (as above), as $W_1< W_2<\cdots<W_m$. Then in the powerful ordering we place, after vertex $w_r$, all vertices $i^{W_1}$ for $i \in V(G)$ in non-decreasing order of colour, i.e., $i^{W_1}$ comes before $j^{W_1}$ if $c(i)<c(j)$. This guarantees that P0 and P1 hold. Next we place in a similar fashion all vertices of the form $i^{W_2}$ and so on, up to vertices of the form $i^{W_m}$ (this guarantees that P2 holds). 
\end{proof}



\begin{lemma}
\label{lem:uAfterPositioni}
Let $i,j \in V(G)$ such that $i \sim j$ in $G$, and assume $i^{W_1} \sim j^{W_2}$ in $\myc^r(G)$ for words $W_1$ and $W_2$ of length $r$ in $u$ and $v$, where $j^{W_2}$ comes after $i^{W_1}$ in a powerful ordering. If $W_1$ has a $u$ in the $s^{th}$ position, where $s > 1$, then for some integer $k$, $1 \le k < s$, $W_1$ has a $v$ in $k^{th}$ position and $W_2$ has a $u$ in the $k^{th}$ position.
\end{lemma}

\begin{proof}
Suppose not.  First, by rules R1--R3, $W_1$ and $W_2$ cannot both have a $u$ as the $m^{th}$ letter for any $1 \le m \le r.$  Thus we may assume 
that the first $s-1$ symbols of $W_2$ are $v$, and so $W_2$ comes before $W_1$ lexicographically.  However, $j^{W_2}$ comes after $i^{W_1}$, so this is a contradiction to property P2 of powerful orderings.  
Hence the result holds.
\end{proof}

\begin{lemma}
\label{lem:wPaths}
Suppose that the vertices of $\myc^r(G)$ are arranged in a powerful ordering.  Then any monotonic path in the ordering beginning at $w_i^{[r-i]}$ has length at most $r+2-i$.
\end{lemma}

\begin{proof}
Let $P$ be such a path, and assume that $w_j^{[r-j]}$ is the last vertex in $P$ that is a $w$ vertex.  By rules R1--R3, the next vertex in $P$ is of the form $a_0^{W_0}$ for some $a_0 \in V(G)$ and $W_0 = [j-1]u[r-j]$.  By Lemma \ref{lem:uAfterPositioni}, there can only be $r-j$ additional vertices in $P$: the vertex after $a_0^{W_0}$ in the path is of the form $a_1^{W_1}$, where $W_1$ contains a $u$ in one of the first $(r-j)$ positions (reading right to left) where $W_0$ contained a $v$; similarly, a vertex $a_2^{W_2}$ after $a_1^{W_1}$ in the path $P$ is such that $W_2$ contains a $u$ in some earlier position still where $W_1$  contains a $v$, etc.  There are only $r-j$ positions before the forced $u$ in $W_0$, so this means there are at most $r-j+1$ vertices in $P$ after the last $w$ vertex in $P$.  On the other hand, any $w_k^{[r-k]}$ occurring before $w_j^{[r-j]}$ in $P$ necessarily has $k < j$ by Lemma \ref{lem:w_i} and by property P4 of powerful orderings.  Proceeding similarly, we find that there can be at most $j-i$ vertices before $w_j^{[r-j]}$ in $P$.  Altogether, this is a total of at most $r+2-i$ vertices in $P$, as desired. 
\end{proof}

\begin{lemma}
\label{lem:tPaths}
Let $G$ be any graph that has an edge, and suppose $\chi(\myc^r(G)) = t$.  In any powerful ordering of the vertices of $\myc^r(G)$, there are no vertices of the form $w_i^{[r-i]}$ in any monotonic paths of length $t$ or longer.
\end{lemma}

\begin{proof}
Since we have a powerful ordering, any monotonic path containing a vertex of the form $w_i^{[r-i]}$ must begin at a vertex of the form $w_j^{[r-j]}$ (by P3), and, by Lemma \ref{lem:wPaths}, this path has length at most $r+1$.  On the other hand, $\chi(G) = t - r \ge 2$, since $G$ contains an edge and $\chi(\myc(G))=\chi(G)+1$.  Hence any monotonic path containing a vertex of the form $w_i^{[r-i]}$ has length at most $r+1 < r+2 \le t$, and so no monotonic path of length $t$ can contain a vertex of the form $w_i^{[r-i]}$. 
\end{proof}

\begin{lemma}
\label{lem:PathLength}
If $G$ is any graph containing an edge and $\chi(\myc^r(G)) = t$, then a powerful ordering of the vertices of $\myc^r(G)$ contains no path of length longer than $t$.
\end{lemma}

\begin{proof}
We proceed by induction on $d$ (at most $r$), where the induction statement is that any powerful ordering of $\myc^d(G)$ with respect to $c$ contains no monotonic path of length $t-r+d$.
First, $G = \myc^0(G)$ contains no monotonic path of length longer than $t-r$ since $\chi(G) = t-r$ and a powerful ordering of $\myc^0(G)$ groups the vertices by colour class by P0.

Now assume for some $0 \le d \le r-1$ that a powerful ordering of the vertices of $\myc^d(G)$ (with respect to $c$) contains no monotonic path of length $t-r+d$ and suppose that there is a powerful ordering of the vertices of $\myc^{d+1}(G)$ (with respect to $c$) that contains a monotonic path $P$ of length $m > t - r + d + 1$.  By Lemma \ref{lem:tPaths}, $P$ cannot contain any vertex of the form $w_i^{[d+1-i]}$ since $\chi(\myc^{d+1}(G)) = t - r + d + 1$.  Let the vertices of $P$ be $i_1^{W_1}, i_2^{W_2}, \ldots, i_m^{W_m}$ when taken monotonically in the powerful ordering, where each word $W_i$ has length $d+1$.  By P2, $W_i \le W_j$ in the lexicographic order for all $i < j$. Since $i_{m-1}^{W_{m-1}} \sim i_m^{W_m}$, either the rightmost letter of $W_m$ is $v$, in which case the rightmost letter of all $W_i$ is $v$, or the rightmost letter of $W_m$ is $u$, in which case the rightmost letter of $W_{m-1}$ is $v$ by rules R2 and R3, and hence the rightmost letter of all $W_i$ is $v$ for all $i \le m-1$ by lexicographic order.  Define $W_i'$ to be the word of length $d$ obtained by deleting the final letter $v$ from each $W_i$, $i \le m-1$.  Note that, since $i_j^{W_j} \sim i_{j+1}^{W_{j+1}}$ and the rightmost letter of $W_j$ is $v$ for $1 \le j \le m - 1$, by the definition of the Mycielskian, $i_1^{W_1'}, \ldots, i_{m-1}^{W_{m-1}'}$ is a path in $\myc^d(G)$.  Moreover, by definition of powerful order, lexicographically $W_1' \le W_2' \le \ldots \le W_{m-1}'$, and so $i_1^{W_1'}, \ldots, i_{m-1}^{W_{m-1}'}$ is in fact a monotonic path of length $m-1 > t-r+d$ in the powerful ordering of $\myc^d(G)$ inherited from the powerful ordering of $\myc^{d+1}(G)$, 
a contradiction to the inductive hypothesis.  Therefore, a powerful ordering (with respect to $c$) of the vertices of $\myc^r(G)$ contains no path of length longer than $t$, as desired.  
\end{proof}


Finally we can prove Theorem~\ref{thm:oddgirth}, which we restate here for convenience.

\begin{quote}
\noindent\textbf{Theorem~\ref{thm:oddgirth}:} Suppose $G$ is nonempty with $\chi(\myc^r(G))=t$, where $t$ is odd, and the length of the shortest odd cycle of $G$ is strictly greater than $t$. Then 
$\ca(\myc^r(G))<t$.
\end{quote}

\begin{proof}[Proof of Theorem~\ref{thm:oddgirth}]
Suppose on the contrary that $\ca(\myc^r(G))\ge t$.  Then there exists a monotone cycle of length at least $t$ in any circular ordering of $V(\myc^r(G))$; in particular, in a powerful ordering of the vertices of $\myc^r(G)$, which exists by Lemma \ref{lemma:powerful}, there must exist a monotonic path $P$ of length at least $t$ such that the first and last vertices are adjacent.  By Lemma \ref{lem:PathLength}, $P$ has length exactly $t$, and by Lemma \ref{lem:tPaths}, $P$ contains no vertex of the form $w_j^{[r-j]}$.  
Let the vertices of $P$ be $i_1^{W_1}, \ldots, i_t^{W_t}$ when taken monotonically in the powerful ordering.  (Note that the vertices $i_1, \ldots, i_t$ are not necessarily distinct.)  By Lemma \ref{lem:w_i}, 
we see that this implies that $i_1 \sim i_2 \sim \cdots \sim i_t \sim i_1$ in $G$, and, in fact, $(i_1, i_2,\ldots, i_t, i_1)$ is a closed walk of length less than $t$ in $G$. 
Since a closed walk of odd length must contain an odd cycle, and the length of the shortest odd cycle of $G$ is strictly greater than $t$, we have a contradiction. Therefore, $\ca(\myc^r(G)) < t$.
\end{proof}

The following is a consequence of the Zig-zag Theorem in~\cite{Simonyi:2006aa}. We use it to prove a corollary to Theorem~\ref{thm:oddgirth} and also Theorem~\ref{thm:camyc}.

\begin{theorem}[Corollary of the Zig-zag Theorem in~\cite{Simonyi:2006aa}]\label{thm:zigzag}
Let $c$ be an arbitrary proper colouring of $\myc^r(G)$ by an arbitrary number of colours, where the colours are linearly ordered. 
Let $t=\chi(\myc^r(G))=\chi(G)+r$. Then  $\myc^r(G)$ contains a complete bipartite subgraph $K_{\lceil \frac{t}{2}\rceil,\lfloor \frac{t}{2}\rfloor}$ such that $c$ assigns distinct colours to all $t$ vertices of this subgraph and these colours appear alternating on the two sides of the bipartite subgraph with respect to their order.
\end{theorem}

\begin{proof}[Proof of Theorem~\ref{thm:camyc}]
Let $t=3+2r$ be the chromatic number of $\myc^{2r}(C_{2n+1})$.
First note that $\omega(\myc^{2r}(C_{2n+1}))=2$ as $\myc^{2r}(C_{2n+1})$ has no triangles.
Consider any circular ordering of the vertices of $\myc^{2r}(C_{2n+1})$. Pick any linear ordering induced by the circular ordering, and let $c$ be the colouring induced by this linear ordering. Then by Theorem~\ref{thm:zigzag}, there is a monotonic cycle of length $t-1=2r+2$ in the circular ordering of $\myc^{2r}(C_{2n+1})$.
Hence $\ca(\myc^{2r}(C_{2n+1})) \ge 2r+2$. By Theorem~\ref{thm:oddgirth}, 
\[
\ca(\myc^{2r}(C_{2n+1})) < \chi(\myc^{2r}(C_{2n+1})) = 2r+3.\]
The result follows.
\end{proof}

\begin{remark}
Theorem~\ref{thm:camyc} can also be proved as follows: by \cite[Corollary 4.1]{Zhu:2001aa}, $\chi_c(C_{2n+1}) = 2 + 1/n$.  By applying both \cite[Theorem 4.3]{Zhu:2001aa} and Theorem \ref{th:caVScirc}, we see that $\ca(\myc^{2r}(C_{2n+1})) \le \chi_c(\myc^{2r}(C_{2n+1})) < \chi(\myc^{2r}(C_{2n+1}))$.  Combined with Theorem \ref{thm:zigzag}, the result follows. 
\end{remark}

\end{document}